\documentclass[12pt,reqno]{amsart}
\usepackage{amssymb,a4wide, xy}
\usepackage{amsmath}
\usepackage{amsfonts}

\xyoption{all}

\numberwithin{equation}{section}
\newtheorem{theorem}{Theorem}[section]

\newtheorem{lemma}[theorem]{Lemma}

\newtheorem{prop}[theorem]{Proposition}

\newtheorem*{con7*}{Conjecture 7*}
\newtheorem{Property}[theorem]{Property}
\newtheorem{Example}[theorem]{Example}
\newtheorem*{Conjecture}{Conjecture}
\newtheorem*{Main Theorem}{Main Theorem}

\theoremstyle{definition}
\newtheorem{Definition}[theorem]{Definition}

\theoremstyle{remark}
\newtheorem*{Remark}{Remark}

\numberwithin{equation}{theorem}

\DeclareMathOperator{\wgd}{w.gl.dim}
\DeclareMathOperator{\Tor}{Tor}
\DeclareMathOperator{\Spec}{Spec}
\DeclareMathOperator{\Max}{Max}
\DeclareMathOperator{\Ker}{Ker}
\DeclareMathOperator{\Img}{Im}
\begin{document}

\title[On a Conjecture on the weak global dimension of Gaussian rings]{On a
Conjecture on the weak global dimension of Gaussian rings}

\author{Guram Donadze \and Viji Z. Thomas}
 \address{G. Donadze\\ Department of Algebra, University of Santiago de Compostela,
15782, Spain.}
 \email {gdonad@gmail.com}
\address{V.Z. Thomas \\ School of Mathematics, Tata Institute of Fundamental Research,
Mumbai, Maharashtra 400005, India.}
\email{vthomas@math.tifr.res.in}
\subjclass[2010]{Primary 13B25,13D05,13F05,16E30 Secondary 16E65}
\keywords{Gaussian rings, weak dimension, content, non-noetherian rings}
\begin{abstract} In \cite{BG}, Bazzoni and Glaz conjecture that the weak global
dimension of a Gaussian ring is $0,1$ or $\infty$. In this paper, we prove their
conjecture in all cases except when $R$ is a non-reduced local Gaussian ring with
nilradical $\mathcal{N}$ satisfying $\mathcal{N}^2=0$.
\end{abstract}

\maketitle

\section{Introduction}
In her Thesis \cite{T}, H. Tsang, a student of Kaplansky introduced Gaussian rings.
Noting that the content of a polynomial $f$ over a commutative ring $R$ is the ideal
$c(f)$ generated by the coefficients of $f$, we now define a Gaussian ring.

\begin{Definition}
A polynomial $f\in R[x]$ is called Gaussian if $c(f)c(g)=c(fg)$ for all $g\in R[x]$.
The ring $R$ is called Gaussian if each polynomial in $R[x]$ is Gaussian.
\end{Definition}

Among other things, H. Tsang (\cite{T}) proved that an integral domain is Gaussian
if and only if it is Pr\"ufer (see Definition \ref{pru}), a result also proved
independently by R. Gilmer in \cite{RG}. Thus Gaussian rings provide another class
of rings extending the class of Pr\"ufer domains to rings with zero divisors. In
\cite{BG}, the authors consider five possible extensions of the Pr\"ufer domain
notion to the case of commutative rings with zero divisors, two among which are
Gaussian rings and rings with weak global dimension (see Definition \ref{weak}) at
most one. The authors also consider the problem of determining the possible values
for the weak global dimension of a Gaussian ring. At the end of their article, the
authors make the following conjecture.

\begin{Conjecture}[Bazzoni-Glaz, \cite{BG}] The weak global dimension of a Gaussian
ring is either $0,1$ or $\infty$.
\end{Conjecture}

 In \cite{SG}, the author shows that the weak global dimension of a coherent
Gaussian ring is either $\infty$ or at most one. She also shows that the weak
global dimension of a Gaussian ring is at most one if and only if it is reduced. So
to prove the conjecture it is enough to show that $\wgd R=\infty$ for all
non-reduced Gaussian rings $R$. Since $\wgd R=\sup\{\wgd R_{\mathfrak{p}} \mid
\mathfrak{p}\in \Spec (R)\}$, it is enough to prove the conjecture for non-reduced
local Gaussian rings. For any non reduced local Gaussian ring $R$ with nilradical
$\mathcal{N}$, either $(i)$ $\mathcal{N}$ is nilpotent or $(ii)$ $\mathcal{N}$ is
not nilpotent. Except when $\mathcal{N}^2=0$, the authors of \cite{BG} prove that
if $R$ satisfies $(i)$, then $\wgd R=\infty$. In this paper we prove that if $R$
satisfies $(ii)$, then $\wgd R=\infty$ (cf. Theorem \ref{main}). In fact the
authors of \cite{BG} do not exclude the case $\mathcal{N}^2=0$ (\cite[Theorem
6.4]{BG}), but we exclude this case for correctness (cf. Section 5, Theorem
\ref{thm5.1}). In this paper we solve the Bazzoni-Glaz conjecture in all cases
except this case. With this in mind, we now state our Main theorem.

\begin{Main Theorem}
Let $R$ be a non-reduced local Gaussian ring with nilradical $\mathcal{N}$. If
$\mathcal{N}^2\neq 0$, then $\wgd (R)=\infty$.
\end{Main Theorem}

In a recent paper \cite{AJK}, the authors have validated the Bazzoni-Glaz conjecture
for the class of rings called fqp-rings. The class of fqp-rings fall strictly
between the classes of arithmetical rings and Gaussian rings.

In Section 3, we consider some homological properties of local Gaussian rings. In
particular we consider local Gaussian rings $(R,\mathfrak{m})$ which are not fields,
with the property that each element of $\mathfrak{m}$ is a zero divisor. In this
case we prove that $\wgd R\geq 3$.

In  \cite[Section 6]{BG}, the authors consider local Gaussian rings $(R,
\mathfrak{m})$ such that the maximal ideal $\mathfrak{m}$ coincides with the
nilradical of $R$. With this set up in Section 4, we prove that if $\mathfrak{m}$ is
not nilpotent, then $\wgd R=\infty$.

Finally in Section 5, we prove our main theorem. As a result of our Main Theorem, we
reduce the Bazzoni-Glaz conjecture to the following {\bf Conjecture}: Let $R$ be a
non-reduced local Gaussian ring with nilradical $\mathcal{N}$. If $\mathcal{N}^2=0$,
then $\wgd (R)=\infty$.

Throughout this paper, $R$ is a commutative ring with unit, $(R,\mathfrak{m})$ is a
local ring(not necessarily Noetherian) with unique maximal ideal $\mathfrak{m}$. We
denote the set of all prime ideals of $R$ by $\Spec(R)$ and the set of all maximal
ideals by $\Max(R)$.

\section{Preliminary Results}

In this section we will recall some definitions and results that we will need in
later sections.

\begin{Definition}
The flat dimension $fd(M)$ is the minimum integer (if it exists) such that there is
a resolution of $M$ by flat $R$ modules
$0\to F_n\to \cdots\to F_1\to F_0\to M\to 0$. If no finite resolution by flat $R$
modules exists for $M$, then we set $fd(M)=\infty$.
\end{Definition}

Now we define the weak global dimension of a ring $R$ denoted as $\wgd(R)$. It is
also sometimes suggestively called as the $\Tor$-dimension.

\begin{Definition}\label{weak}
$\wgd(R)$=$\sup$\{$fd(M) \mid M$ is an $R$-module\}.
\end{Definition}

Recall that $\wgd(R)$=$\sup$\{$d \mid \Tor_d(M,N)\neq 0$ for some $R$-modules
$M,N$\}. The $\wgd(R)\leq 1$ if and only if every ideal of $R$ is flat, or
equivalently, if and only if every finitely generated ideal of $R$ is flat.

We now define Pr\"ufer domain. A Noetherian Pr\"ufer domain is a Dedekind domain.

\begin{Definition}\label{pru}
A domain is Pr\"ufer if every non-zero finitely generated ideal is invertible.
\end{Definition}

L. Fuchs introduced the class of arithmetical rings in \cite{F}.

\begin{Definition}
A ring $R$ is arithmetical if the lattice of the ideals of $R$ is distributive.
\end{Definition}

In \cite{J}, the author characterized arithmetical rings by the property that in
every localization at a maximal ideal, the lattice of the ideals is linearly ordered
by inclusion. Hence in a local arithmetical ring, the lattice of the ideals is
linearly ordered by inclusion. Thus local arithmetical rings provide another class
of rings extending the class of valuation domains to rings with zero-divisors.

The next theorem appears in Tsang's (see \cite{T}) unpublished thesis.

\begin{theorem}[\cite{T}]\label{thm2.1}
Let $R$ be a Gaussian ring. If $R$ is local, then
\begin{itemize}
  \item[(i)] $R$ is Gaussian if and only if $R_{\mathfrak{m}}$ is Gaussian for all
$\mathfrak{m}\in \Max(R)$;
   \item[(ii)] $R$ is Gaussian if and only if $R_{\mathfrak{p}}$ is Gaussian for all
$\mathfrak{p}\in \Spec(R)$;
   \item[(iii)] the prime ideals of $R$ are linearly ordered under inclusion; and
 \item[(iv)] the nilradical of $R$ is the unique minimal prime ideal of $R$.
\end{itemize}
\end{theorem}

We will need several equivalent characterizations of local Gaussian rings, which we
now state.

\begin{theorem}\label{thm2.2}
Let $(R,\mathfrak{m})$ be a local ring with maximal ideal $\mathfrak{m}$. The
following conditions are equivalent.
\begin{itemize}
  \item[(i)] $R$ is a Gaussian ring;
   \item[(ii)] If $I$ is a finitely generated ideal of $R$ and $(0:I)$ is the
annihilator of $I$, then $I/{I\cap(0:I)}$ is a cyclic $R$-module;
   \item[(iii)] Condition $(ii)$ for two generated ideals;
 \item[(iv)] For any two elements $a,b\in R$, the following two properties hold:
 \begin{itemize}
   \item[(a)] $(a,b)^2=(a^2)$ or $(b^2)$;
  \item[(b)] If $(a,b)^2=(a^2)$ and $ab=0$, then $b^2=0$.
 \end{itemize}
 \item[(v)] If $I=(a_1,a_2,\ldots,a_n)$ is a finitely generated ideal of $R$, then
$I^2=(a_{i}^2)$ for some $1\leq i\leq n$.
\end{itemize}
\end{theorem}

The implication $(iv)\Rightarrow (i)$ was noted by Lucas in \cite{L} and the rest of
Theorem \ref{thm2.2} was proved by Tsang in \cite{T}. The next two results can be
found in \cite{BG}.

\begin{theorem}[\cite{BG}]\label{thm2.3}
Let $(R,\mathfrak{m})$ be a local Gaussian ring and let $D=\{x\in R \mid x^2=0\}$.
The following hold:
\begin{itemize}
  \item[(i)] $D$ is an ideal of $R$, $D^2=0$, and $R/D$ is an arithmetical ring;
   \item[(ii)] For every $a\in R$, $(0:a)$ and $D$ are comparable and $D\subseteq
Ra+(0:a)$;
   \item[(iii)] If $a\in m\setminus D$, then $(0:a)\subseteq D$;
 \item[(iv)] Let $\mathfrak{m}$ be the nilradical of $R$. If $\mathfrak{m}$ is not
nilpotent, then $\mathfrak{m}=\mathfrak{m^2}+D$ and
$\mathfrak{m^2}=\mathfrak{m^3}$.
 \end{itemize}
\end{theorem}

\begin{prop}\label{prop2.1}
Let $(R,\mathfrak{m})$ be a local Gaussian ring. If $\mathfrak{m}$ is non-zero and
nilpotent, then $\wgd  \mathfrak{m}=\infty$.
\end{prop}

\section{Some results on local Gaussian rings}

It is well known that if the $\wgd R=n$, then there exists a cyclic $R$-module, say
$R/I$ such that $\wgd R/I=n$. In the next lemma we show that this cyclic module can
be chosen with some additional properties.

\begin{lemma}\label{lemma2}
Let $R$ be local Gaussian ring with $\wgd(R)=n$ and let $I$ be an ideal of $R$. If
the $\wgd_R(R/I)=n$, then there exists an ideal $J\subset R$
such that $\wgd_R(R/J)=n$ and $J\supseteq I+D$.
\end{lemma}
\begin{proof} Let $M$ be an $R$-module such that $\Tor_n (R/I, M)\neq 0$. Suppose
the lemma is not true. Then we prove that the natural projection $R/I \to
R/(I+x_1R+\cdots +x_mR)$ induces an inclusion
\begin{equation}\label{tor1}
 \Tor_n (R/I, M)\hookrightarrow \Tor_n (R/(I+x_1R+\cdots +x_mR), M)
\end{equation}
 for any finite subset $\{x_1, \dots, x_m\}\subset D$. Set $I_0=I$ and define $I_p$
inductively as $I_p=I_{p-1}+x_p R$ for all $1\leq p\leq m$.
 We have the following short exact sequence
\[
 0\to (x_pR+I_{p-1})/I_{p-1} \to R/I_{p-1} \to R/I_p \to 0
\]
 for all $1\leq p\leq m$. The homomorphism $f: R \to (x_pR+I_{p-1})/I_{p-1}$ defined
by $f(r)=rx_p+I_{p-1}$ for all $r\in R$ induces an isomorphism $R/\Ker f \cong (x_p
R+I_{p-1})/I_{p-1}$. Furthermore we have that $\Ker f\supset I_{p-1}+ (0:x_p)
\supset I+D$. If $\Tor_n ((x_pR+I_{p-1})/I_{p-1}, M)\neq 0$, then the lemma is true
with $J=\Ker f$. So assume that $\Tor_n ((x_pR+I_{p-1})/I_{p-1}, M)= 0$. In this
case the natural projection $R/I_{p-1} \to R/I_p$ induces an inclusion
$\Tor_n (R/I_{p-1}, M)\hookrightarrow \Tor_n (R/I_p, M)$ for all $1\leq p\leq m$,
proving (\ref{tor1}).

 Now let $\mathcal{X}$ denote the following class of ideals: $J\in \mathcal{X}$ iff
$J\subset D$ and $J$ is finitely generated. Then
\begin{equation}\label{directlim}
\varinjlim_{J\in \mathcal{X}} \Tor_n(R/(I+J),M)=\Tor_n(R/(I+D),M)
\end{equation}

Using (\ref{tor1}) and (\ref{directlim}), we obtain an inclusion $\Tor_n (R/I,
M)\hookrightarrow \Tor_n (R/(I+D), M)$. Thus $\Tor_n (R/(I+D), M)\neq 0$ and the
lemma is proved.
\end{proof}

The next lemma is an immediate consequence of the long exact sequence of $\Tor$
groups applied to the given short exact sequence. We note it here for the readers
convenience.

\begin{lemma}\label{lemma3}
Let $R$ be a commutative (not necessarily local Gaussian) ring. Let $M_1$ and $M_2$
be $R$-modules and $f:M_1\to M_2$
be an injective homomorphism. If the $\wgd(R)=n$, then $f_*:\Tor_n(M_1, -) \to
\Tor_n(M_2, -)$ is also injective.
\end{lemma}
\begin{proof} The proof is a direct consequence of the long exact sequence of $\Tor$
groups applied to the given short exact sequence and the fact that
$\Tor_{n+1} (M_2/M_1, -)=0$.
\end{proof}

\begin{lemma}\label{lemma4}
Let $(R,\mathfrak{m})$ be a local Gaussian ring. If the $\wgd(R)=n$, \newline
then $\Tor_n (R/D, - )=0$ for all $n\geq 1$.
\end{lemma}
\begin{proof} If the lemma is not true, then there exists a module $M$ such that
$\Tor_n (R/D, M)\neq 0$. Consider a free resolution
of $M$: $\cdots \xrightarrow{\partial_{n+2}} R^{X_{n+1}}\xrightarrow{\partial_{n+1}}
R^{X_n}\xrightarrow{\partial_{n}}
\cdots \xrightarrow{\partial_1} R^{X_0}\xrightarrow{\partial_0} M$, where $X_i$ are
sets. By assumption
$\Tor_n (R/D, M)=\Ker(\overline{\partial_n})/ \Img(\overline{\partial_{n+1}})\neq
0$, where $\overline{\partial_i}$ is the natural homomorphism
$\overline{\partial_i}:(R/D)^{X_i}\to (R/D)^{X_{i-1}}$ obtained after tensoring the
above resolution by $R/D$ for all $i\in \mathbb{N}$. Since $\Tor_n (R/D, M)\neq 0$,
there exists a $\overline{w}\in \Ker(\overline{\partial_n})$ such that
$\overline{w}\notin \Img(\overline{\partial_{n+1}})$. Let $w$ be the representative
of $\overline{w}$ in $R^{X_n}$. Hence $\partial_n(w)\in D^{X_{n-1}}$. Let
$\lambda_1, \dots,\lambda_m\in D$ be the finitely many non-zero entries of
$\partial_n(w)$. Now we consider two cases.

{\bf Case 1}. There exists an $a\in \mathfrak{m}\setminus D$ such that
$a\lambda_j=0$ for all $1\leq j\leq m$.\\
Define a homomorphism $f:R/D \to R/a D$ which is multiplication by $a$. Using
Theorem \ref{thm2.3}$(iii)$, it follows that $(0:a)\subset D$. This gives the
injectivity of $f$. Therefore by Lemma \ref{lemma3}, $f_*:\Tor_n (R/D, M) \to \Tor_n
(R/(aD), M)$ is injective and hence $f_*(\overline{w})\neq 0$. It is easy to verify
that $aw$ is a representative of $f_*(\overline{w})$ in $R^{X_n}$.
Since $a\in (0:\lambda_j)$ for all $1\leq j\leq m$, we obtain that
$\partial_n(aw)=a\partial_n(w)=0$. This would imply that $f_*(\overline{w})=0$, a
contradiction.

{\bf Case 2}. For all $a\in R\setminus D$ at least one $a\lambda_j\neq 0$.\\
We have an injective homomorphism $g: R/D\to R^m$ defined by $g(1)=(\lambda_1, \dots
, \lambda_m)$. By Lemma \ref{lemma3}, the induced homomorphism
$g_*:\Tor_n (R/D, M) \to \Tor_n (R^m,M)$ is injective. This is a contradiction as
$\Tor_n (R^m,M)=0$.
\end{proof}

\begin{lemma}\label{singleann}
Let $(R,\mathfrak{m})$ be a local Gaussian ring such that each element of
$\mathfrak{m}$ is a zero divisor. If $\lambda_1, \dots, \lambda_n\in \mathfrak{m}$,
then there exists a non-trivial element $a\in \mathfrak{m}$ such that $a\lambda_j=0$
for all $1\leq j\leq n$.
\end{lemma}

\begin{proof}
We divide the proof into two cases.

{\bf Case 1}: $\lambda_1,\dots,\lambda_n\in D$.

First we claim that $D\neq 0$. Towards that end, let $0\neq a,b\in \mathfrak{m}$
with $ab=0$. By Theorem \ref{thm2.2}$(iv)$, we have either $a^2=0$ or $b^2=0$. Thus
$D\neq 0$. So take any $d\in D\setminus 0$. Using Theorem \ref{thm2.3}$(i)$, we
obtain that $d\lambda_i=0$ for all $1\leq i\leq n$.

{\bf Case 2}:  There exist $j\in \{1,2,\ldots,n\}$ such that $\lambda_j\notin D$.

By Theorem \ref{thm2.3}$(iii)$, it follows that $(0:\lambda_j)\subseteq D$. Set
$I=(\lambda_1,\ldots, \lambda_n)$. So $(0:I)\subset (0:\lambda_j)\subseteq D$. Using
Theorem \ref{thm2.2}$(ii)$, we obtain that $I/{I\cap (0:I)}$ is a cyclic $R$-module,
say its generator is $\lambda$. Hence we can write $\lambda_i=r_i\lambda+d_i$ for
all $1\leq i\leq n$, where $d_i\in I\cap (0:I)$ and $r_i\in R$. Observe that
$\lambda\in \mathfrak{m}\setminus D$. Choose any $d\in (0:\lambda)\setminus 0$.
Using Theorem \ref{thm2.3}$(iii)$, it follows that $d\in D$. Multiplying the
equation expressing $\lambda_i$ in terms of $\lambda$ with $d$, we obtain
$d\lambda_i=dr_i\lambda + dd_i$ for all $1\leq i\leq n$. Using  Theorem
\ref{thm2.3}$(i)$, we obtain that $dd_i=0$. Thus $d\lambda_i=0$ for all $1\leq i\leq
n$.

\end{proof}

\begin{lemma}\label{lemma5}
Let $(R,\mathfrak{m})$ be a local Gaussian ring with $\wgd(R)=n$. If each element of
$\mathfrak{m}$ is a zero divisor, then
$\Tor_n (R/\mathfrak{m}, - )=0$ for all $n\geq 1$.
\end{lemma}
\begin{proof} The proof of this Lemma follows by substituting $\mathfrak{m}$ for $D$
in Lemma \ref{lemma4}. As a result of Lemma \ref{singleann}, the proof of lemma
\ref{lemma5} falls under Case 1 of Lemma \ref{lemma4}.
\end{proof}

\begin{prop}\label{prop1.1}
Let $(R,\mathfrak{m})$ be a local Gaussian ring. If $\mathfrak{m}\neq 0$ and each
element of $\mathfrak{m}$ is a zero divisor, then
$\wgd (R)\geq 3$.
\end{prop}
\begin{proof} If $\mathfrak{m} =D$, then Proposition \ref{prop2.1} implies that
$\wgd (R) = \infty$. If $\mathfrak{m}\neq D$, then take any $x\in
\mathfrak{m}\setminus D$ and consider
the following resolution of $R/xR$:
\[
0\to (0:x)\to R\xrightarrow{\sigma}R\xrightarrow{\tau}R/xR ,
\]
where $\tau$ is the natural projection and $\sigma(r)=xr$ for all $r\in R$. If $\wgd
(R)<3$, then $(0:x)$ must be flat. Thus
it suffices to show that $(0:x)$ is not flat. We will use the fact that if $M$ is a
flat $R$ module then $I\otimes M=IM$ for all ideals $I\subset R$. Set $I=xR$ and
$M=(0:x)$ and observe that $IM=0$. Hence it suffices to show that $I\otimes
(0:x)\neq 0$. Since $x\notin D$, Theorem \ref{thm2.3}$(iii)$ implies that
$(0:x)\subset D$. Define a homomorphism $\theta : I\otimes (0:x)\to (0:x)$ as
follows: if $a\in I$ and $b\in (0:x)$, then set $\theta(a\otimes b)=rb$, where $r\in
R$
is such that $a=xr$. If there is another $r'\in R$ such that $a=xr'$, then
$(r-r')\in (0:x)$ which implies that $(r-r')b=0$.
Taking into account the last remark, it is easy to check that $\theta$ is well
defined. Moreover the homomorphism
$\theta':(0:x)\to I\otimes (0:x)$ defined by $\theta'(c)=x\otimes c$ for all $c\in
(0:x)$ is an inverse of $\theta$. Hence we have an isomorphism
$\theta:I\otimes (0:x)\cong (0:x)$ which shows that $I\otimes (0:x)\neq 0$, proving
that $(0:x)$ is not flat.
\end{proof}

\section{Local Gaussian rings with nilradical being the maximal ideal}

Let $R$ be a local Gaussian ring which admits the following property:
\begin{Property}\label{hypo1}
For all $x\in D\setminus 0$, $(0:x)$ is not cyclic modulo $D$. In other words there
is no $a\in R\setminus D$ such that $(0:x)=aR+D$.
\end{Property}

\begin{lemma}\label{flat} Let $(R,\mathfrak{m})$ be a local Gaussian ring such that
each element of $\mathfrak{m}$ is a zero divisor. If $R$ satisfies Property
(\ref{hypo1}) and $\mathfrak{m}\neq D$ , then

\begin{itemize}
\item[(i)] $\mathfrak{m}=\mathfrak{m}^2+D$;
\item[(ii)] for any finitely generated ideal $J\subset \mathfrak{m}$ there exist
$x\in \mathfrak{m}$ such that $J^2\subset x^2R$ and $x^2\notin D$;
\item[(iii)] $\mathfrak{m}^2$ is flat.
\end{itemize}
\end{lemma}

\begin{proof} (i): Let $a\in \mathfrak{m}\setminus D$. Since every element of
$\mathfrak{m}$ is a zero divisor, there exists $x\in D\setminus 0$ such that $ax=0$.
By Property (\ref{hypo1}), $(0:x)\neq aR+D$. So there exists some $b\in
\mathfrak{m}$ such that $b\in (0:x)$ and $b\notin aR+D$. Theorem \ref{thm2.3}$(i)$
implies that $R/D$ is a local arithmetical ring. So $a\in b R+D$ and hence $a=b r +
d$ for some $r \in R$ and $d\in D$. Moreover $b\notin aR+D$ which implies that $r$
is not a unit and hence $r\in \mathfrak{m}$. Thus $a\in \mathfrak{m}^2+D$.

(ii): First we will show that if $x^2\in D$ for all $x\in \mathfrak{m}$, then
$\mathfrak{m^2}\subset D$. Towards that end let $z\in \mathfrak{m^2}$. Such a $z$ is
of the form $z=\sum_{i=1}^n x_iy_i$, where $x_i,y_i\in m$ for all $1\leq i\leq n$.
Using Theorem \ref{thm2.2}$(iv)$, if follows that
$(x_i,y_i)^2=(x_i^2)\;\text{or}\;(y_i^2)$ for all $1\leq i\leq n$. This shows that
$x_iy_i\in D$ for all $1\leq i\leq n$. Recalling that $D$ is an ideal of $R$, it
follows that $z\in D$. Hence we have proved that $\mathfrak{m^2}\subset D$. By $(i)$
this would imply that $\mathfrak{m}=D$, a contradiction. Thus there exists an $x\in
m$ such that $x^2\notin D$. By Theorem \ref{thm2.2}$(v)$, for any finitely generated
ideal $J$ we have $J^2=y^2R$ for some $y\in J$. If $y^2\notin D$ then we are done.
If $y^2\in D$, choose any $x\in \mathfrak{m}$ with $x^2\notin D$ and observe that
$y^2\in x^2R$. Thus $J^2\subset x^2 R$.

(iii): To prove that $\mathfrak{m}^2$ is flat over $R$, we show that for any ideal
$I\subset R$, the natural homomorphism $f: I\otimes \mathfrak{m}^2 \to
\mathfrak{m}^2$ is injective. Assume that $w\in I\otimes \mathfrak{m}^2$ is such
that $f(w)=0$. Set $w = \sum_{i=1}^k  z_i\otimes x_i y_i$, where $z_i\in I$ and
$x_i,y_i\in \mathfrak{m}$. By (ii) there exist $x\in \mathfrak{m}$ such that
$x^2\notin D$ and $x_iy_i\in x^2 R$ for all $1\leq i\leq n$. Put $x_iy_i=x^2r_i$,
where $r_i\in R$. Then $w=z\otimes x^2$, where $z =\sum_{i=1}^k z_ir_i \in I$.
Hence $f(z \otimes x^2)=0 \Leftrightarrow zx^2=0$. If $z=0$ then $w=0$ and the proof
is finished.
So assume that $z\neq 0$. Using Theorem \ref{thm2.3}$(iii)$, we obtain that
$(0:a)\subseteq D$ for all $a\in \mathfrak{m}\setminus D$. Since $x^2\in
\mathfrak{m}\setminus D$, it follows that $z\in D$ and either $zx = 0$ or $zx\neq
0$. If $zx = 0$, then $z\in D\setminus 0$ and $x\in (0:z)$.
It follows by Property (\ref{hypo1}) that $(0:z)\neq xR+D$. So there exists $y\in
\mathfrak{m}$ such that $y\in (0:z)$ and $y\notin xR+D$. By Theorem
\ref{thm2.3}$(i)$, we obtain that $R/D$ is a local arithmetical ring. Hence
$(y)\not\subset (x)$. So $x=cy+d'$, where $c\in \mathfrak{m}$ and $d'\in D$.
Computing $w$, we obtain
\[
w = z\otimes x^2 = z\otimes (cy+d')^2=z\otimes (c^2y^2+2cyd'+d'^2)=zy^2\otimes
c^2+zd'\otimes 2cy+z\otimes d'^2 .
\]
Noting that $d'^2,zd'\in D^2=0$ and that $zy^2=0$, we obtain $w=0$. If $zx\neq 0$,
we have $zx\in D\setminus 0$ and $x\in (0:zx)$. By (\ref{hypo1}), there exists
$h\in \mathfrak{m}$ such that $h\in (0:zx)$ and $h\notin xR+D$. Using the same
argument as above, there exists an $a \in \mathfrak{m}$ such that $x=ah+d''$.
Observing that $zd'',d''^2\in D^2$ we obtain that $w = z\otimes x^2=z\otimes
(ah+d'')^2=z\otimes a^2h^2$. Furthermore, by (i) we can write $a=b+d$ where $b\in
\mathfrak{m}^2$ and $d\in D$. Therefore
\begin{equation}\label{tensor}
w=z\otimes (ah^2(b+d))= z\otimes (ah^2b)+z\otimes (ah^2d) =(zah^2)\otimes
b+(zd)\otimes (ah^2).
\end{equation}
Substituting $0=zd\in D^2$ and $ah=x-d''$ in (\ref{tensor}) and recalling that $h\in
(0:zx)$, we obtain $w=(zxh)\otimes b-zhd''\otimes b=0$.
\end{proof}

\begin{lemma}\label{lemma1.2}
Let $(R,\mathfrak{m})$ be a local Gaussian ring with $\wgd(R)=n$. Let $\mathfrak{m}$
be the nilradical of $R$. If $R$ satisfies Property (\ref{hypo1})
and $\mathfrak{m}$ is not nilpotent, then $\Tor_{n-1}(R/\mathfrak{m} , -)=0$ for all
$n\geq 3$.
\end{lemma}
\begin{proof} By applying the long exact sequence of $\Tor$ groups to the short
exact sequence $0\to \mathfrak{m^2}\to R\to R/{\mathfrak{m^2}}\to 0$ and using Lemma
\ref{flat}$(iii)$, it follows that $\Tor_k(R/\mathfrak{m^2}, -)=0$ for all $k\geq
2$. If $\mathfrak{m} =\mathfrak{m^2}$, then the lemma is proved.
So assume that $\mathfrak{m} \neq \mathfrak{m^2}$. Observing that $\mathfrak{m} /
\mathfrak{m^2}$ is a
vector space over $R/\mathfrak{m}$, we obtain that $\mathfrak{m} /\mathfrak{m^2} =
\bigoplus R/\mathfrak{m}$. Consider the short exact sequence
$0\to \mathfrak{m} / \mathfrak{m^2} \to R / \mathfrak{m^2}\to R / \mathfrak{m} \to 0$.
Consider the following segment of the corresponding long exact sequence of $\Tor$
groups
\[
 \Tor_n(R / \mathfrak{m} , -)\to \Tor_{n-1}(\mathfrak{m}/ \mathfrak{m^2}, -)\to
\Tor_{n-1}(R / \mathfrak{m^2}, -) .
\]
Using Lemma \ref{lemma5} and Lemma \ref{flat}, we obtain that
$\Tor_{n-1}(\mathfrak{m} / \mathfrak{m^2}, -)=0$. This shows that
$\Tor_{n-1}(\bigoplus R/\mathfrak{m}, -)=0$. Recalling that $\Tor_{n-1}(\bigoplus
R/\mathfrak{m}, -)=\bigoplus \Tor_{n-1}(R/\mathfrak{m}, -)$ proves our lemma.
\end{proof}

The idea of the next lemma is taken from \cite{O}, but we give a slightly different
proof.

\begin{lemma}\label{lemma6}
Let $(R,\mathfrak{m})$ be a local arithmetical ring with nilradical $\mathfrak{m}$.
For any $x\in \mathfrak{m}\setminus 0$, if $(0:x)=I$ then $(0:I)=(x)$.
\end{lemma}
\begin{proof}  Clearly $(x)\subseteq (0:I)$. We want to show that $(0:I)\subseteq
(x)$. Towards that end assume that there exists a $z\in (0:I)$ such that $z\notin
(x)$. Recalling that the ideals in a local arithmetical ring are linearly ordered
under inclusion, we obtain that $x=\lambda z$ where $\lambda\in \mathfrak{m}$. Hence
$\lambda\notin I$ which implies that $I\subset (\lambda)$. By induction on $k$, we
will show that $I\subset (\lambda^k)$ for all $k\in \mathbb{N}$. The case $k=1$ is
obvious. Let $b\in I$ be arbitrary. Since $I\subset (\lambda)$ there exists a $t\in
\mathfrak{m}$ such that $b=\lambda t$. Notice that we have $0=zb=z\lambda t=xt$.
Hence $t\in I=(0:x)$. By induction hypothesis $I\subset (\lambda^k)$. So
$t=\lambda^kt_1$ where $t_1\in \mathfrak{m}$. Hence $b=\lambda^{k+1} t_1$, where
$t_1\in \mathfrak{m}$. Thus $I\subset (\lambda^{k+1})$ for all $k\in \mathbb{N}$.
Since $\lambda$ is nilpotent, we obtain that $I=0$, a contradiction.
\end{proof}

In what follows let $R'=R/D$ and $\mathfrak{m'}= \mathfrak{m}/D$. Recall that if $R$
is a local Gaussian ring, then $R'$ is a local arithmetical ring by Theorem
\ref{thm2.3}$(i)$.

\begin{lemma}\label{lemma7}
Let $(R,\mathfrak{m})$ be a local Gaussian ring with nilradical $\mathfrak{m}$. If
$\wgd (R)=n\geq 1$, then
\begin{itemize}
  \item[(i)] there is a non trivial element $x\in \mathfrak{m'}$ such that $\wgd_R
(R'/xR')=n$;
 \item[(ii)] for any non trivial element $z\in \mathfrak{m'}$ and  ideal $J\subset
R'$ such that $z\in J$, $zR'\neq J$, the natural projection $R'/zR'\to R'/J$
induces a trivial map
$0:\Tor_n(R'/zR', -)\to \Tor_n(R'/J, -)$;
  \item[(iii)] $\wgd_R (R'/zR')=n$ for any non trivial element $z\in \mathfrak{m'}$.
\end{itemize}
\end{lemma}
\begin{proof} (i): There is an ideal $I\subset R$ such that $\wgd_R (R/I)=n$.
Without loss of generality one can assume that $D\subset I$
(see Lemma \ref{lemma2}). Hence $\wgd_R (R'/I')=n$, where $I'=I/D$. Using Lemma
\ref{lemma4}, we obtain that $I'\neq 0$. Let $\mathcal{X}$ denote the
following class of ideals: $J\in \mathcal{X}$ iff $J\subset I'$ and $J$ is finitely
generated. Since
$\Tor_n(R'/I', -)= \varinjlim_{J\in \mathcal{X}}\Tor_n(R'/ J, -)$ and $\wgd_R
(R'/I')=n$, there exist $x_1,\dots, x_m\in I'$ such that $\wgd_R (R'/ (x_1',\dots ,
x_m))=n$. Since $R'$ is local arithmetical, $(x_1,\dots ,x_m)=(x_i)$ for some $1\leq
i\leq m$. Thus the first part of the lemma is proved.

(ii):  Let $I=(0: z)\subset R'$. Using Lemma \ref{lemma6}, we obtain that
$(0:I)=zR'$. This implies that $(0:I)\subset J$ and $(0:I)\neq J$. Hence there
exists
$y\in I$ such that $(0:y)\subset J$ and $(0:y)\neq J$. Thus we have the inclusions
$zR' \subset (0:y)\subset J$ which give rise to the natural projections
$R'/zR' \to R'/(0:y)\to R'/J$. Using Lemma's \ref{lemma3} and \ref{lemma4}, we
obtain that $\Tor_n(R'/(0:y), -)=0$, since $R'/(0:y)\cong yR' \subset R'$. Hence the
composition
of the following maps $\Tor_n(R'/zR', -) \to \Tor_n(R'/(0:y), -)\to \Tor_n(R'/J, -)$
is trivial.

(iii): By (i) we have a non trivial element $x\in \mathcal M'$ such that $\wgd_R
(R'/xR')=n$. For any non trivial element $z\in \mathfrak{m'}$, either $z\in xR'$ or
$x\in zR'$. If $z\in xR'$ and $z\neq x$, then there exists $\lambda \in
\mathfrak{m'}$ such that $z=\lambda x$. Define a map $\alpha : R'/xR'\to R'/zR'$ by
$\alpha(r+xR')= \lambda r+zR'$ for all $r\in R'$. Since $x\notin (0:\lambda)$, it
follows that $(0:\lambda)\subset xR'$. This shows that $\alpha$ is injective. Using
Lemma \ref{lemma3}, we obtain that $\alpha$ induces an inclusion $\Tor_n(R'/xR',
-)\hookrightarrow \Tor_n(R'/zR', -)$. Thus $\Tor_n(R'/zR', -)\neq 0$.

In the case when $x\in zR'$ and $x\neq z$, there exists $\lambda'\in \mathfrak{m'}$
such that $x=\lambda' z$. Define a map $\sigma : R'/zR'\to R'/xR'$ by
$\sigma(r+zR')=r +xR'$ for all $r\in R'$. Since $z\notin (0:\lambda')$, we obtain
that $(0:\lambda')\subset zR'$. Thus $\sigma$ is injective.
Consider the short exact sequence
 \[0 \rightarrow R'/zR'\xrightarrow{ \ \sigma \ } R'/xR'\xrightarrow{ \ \tau \ }
R'/\lambda' R' \rightarrow 0\,,\]
where $\tau $ is the natural projection. Observe that $xR'\subset \lambda' R'$ and
$xR'\neq \lambda' R'$. Using $(ii)$, we see that $\tau$ induces the trivial map
$0:\Tor_n( R'/xR', -)\to \Tor_n( R'/\lambda'R', -)$. Therefore $\sigma$ induces an
epimorphism \[\Tor_n( R'/zR', -)\twoheadrightarrow \Tor_n( R'/xR', -)\, , \]
which implies that $\Tor_n( R'/zR', -)\neq 0$.
\end{proof}

Let $\deg(r)$ denote the degree of nilpotency of $r\in R$. Noting that the
nilpotency degree of an element $r\in R$ is the smallest $k\in \mathbb{N}$ such that
$r^k=0$, we state our next lemma.

\begin{lemma}\label{deg}
Let $(R, \mathfrak{m})$ be a local Gaussian ring with nilradical $\mathfrak{m}$ and
let $\lambda\in \mathfrak{m}$. If $\mathfrak{m}$ is not nilpotent, then there exists
$z\in \mathfrak{m}$ such that $\deg(z) > \deg(\lambda)$.
\end{lemma}
\begin{proof}
Let $\deg(\lambda)=n$. Suppose the lemma is not true, then $\deg(z)\leq
\deg(\lambda)$ for all $z\in \mathfrak{m}$. Now we will show that
$\mathfrak{m}^n=0$. This will give us a contradiction, as $\mathfrak{m}$ is not
nilpotent. Towards that end, let $z_1,\ldots,z_n\in \mathfrak{m}$ and consider
$I=(z_1,\ldots,z_n)$. Using Theorem \ref{thm2.2}$(ii)$, we can write $z_i=r_iz+d_i$
for some $z\in I$, $r_i\in R$ and $d_i\in I\cap (0:I)\subset (0:z_i)$ for all $1\leq
i\leq n$. So
\begin{equation}\label{prod}
z_1z_2\cdots z_n=\prod_{i=0}^n (zr_i+d_i)
\end{equation}
After expanding the right hand side of (\ref{prod}), observe that every term of the
expansion except the term $d_1\cdots d_n$ contains a $z$ and some $d_i$, where
$1\leq i\leq n$. Using Theorem \ref{thm2.3}$(i)$, it follows that $d_1\cdots d_n=0$.
Since $d_iz=0$ for all $1\leq i\leq n$, every term in the expansion is zero. Thus
$\mathfrak{m}^n=0$, a contradiction.

\end{proof}

\begin{lemma}\label{lemma8}
Let $(R,\mathfrak{m})$ be a local Gaussian ring with nilradical $\mathfrak{m}$. If
$\wgd (R)=n\geq 1$ and $\mathfrak{m}$ is not
nilpotent, then $\Tor_n(R'/a\mathfrak{m'}, -)=0$ for all non trivial $a\in
\mathfrak{m'}$.
\end{lemma}
\begin{proof} If $a\mathfrak{m'}=0$, then Lemma \ref{lemma4} gives the desired
result. So assume that $a\mathfrak{m'}\neq 0$.
We claim that $a\mathfrak{m'}$ is not a finitely generated ideal. Suppose
$a\mathfrak{m'}$ is a finitely generated ideal. Since $R'$ is an arithmetical ring,
there exists an element $\lambda \in \mathfrak{m'}$ such that $a\mathfrak{m'} =
a\lambda R'$. Let $\deg(x)$ denote
the degree of nilpotency of $x$ for all $x\in \mathfrak{m'}$. Since $\mathfrak{m}$
is not nilpotent, $\mathfrak{m'}$
 is not nilpotent. By Lemma \ref{deg}, there exists $z\in \mathfrak{m'}$ such that
$\deg(z)>\deg(\lambda)$.
  Observe that $\lambda \in z\mathfrak{m'}$, i.e. $\lambda =zh$ for some $h\in
\mathfrak{m'}$. Hence $az\neq 0$. Furthermore $1-hr$ is a unit for all $r\in R'$.
This implies that
$az-a\lambda r=a(z-zhr)=az(1-hr)\neq 0$. Thus $az\notin a\lambda R'$, a contradiction.

Now let $\mathcal{X}$ be the following class of ideals: $J\in \mathcal{X}$ iff
$J\subset a\mathfrak{m'}$ and $J$ is finitely generated. Then
$\Tor_n(R'/a\mathfrak{m'}, -)= \varinjlim_{J\in \mathcal{X}} \Tor_n(R'/ J, -)$.
Since $R'$ is a local arithmetical ring, there exists $c\in \mathfrak{m'}$ such that
$I=cR'$ for all $I\in \mathcal{X}$. As $a\mathfrak{m'}$ is not finitely generated,
$I\neq a\mathfrak{m'}$. Using Lemma \ref{lemma7}$(ii)$, we obtain that the natural
projection $R'/I\to R'/a\mathfrak{m'}$ induces a trivial homomorphism $0:\Tor_n(R'/
I, -)\to \Tor_n(R'/a\mathfrak{m'}, -)$.
Thus the canonical homomorphism
$\Tor_n(R'/ I, -)\to \varinjlim_{J\in \mathcal{X}}\Tor_n(R'/ J, -)$
is trivial for all $I\in \mathcal{X}$. This implies $\varinjlim_{J\in
\mathcal{X}}\Tor_n(R'/ J, -)=0$.
\end{proof}

\begin{theorem}\label{thm1}
Let $(R,\mathfrak{m})$ be a local Gaussian ring with nilradical $\mathfrak{m}$. If
$\mathfrak{m}$ is not nilpotent, then $\wgd (R)=\infty$.
\end{theorem}
\begin{proof} Suppose the theorem is not true, then the $\wgd (R)=n<\infty$. Using
Proposition \ref{prop1.1}, we obtain that $n\geq 3$. We divide
the proof into two cases.

{\bf Case 1.} $R$ does not satisfy Property (\ref{hypo1}).

Hence there exists $x\in D\setminus 0$ and $a\in R\setminus D$ such that
$(0:x)=aR+D$. Thus we have an isomorphism $R/(aR+D) \cong xR$. Using Lemma
\ref{lemma3}  and noting that $xR\subset R$, we obtain an inclusion $\Tor_n
(R/(aR+D), -)\hookrightarrow \Tor_n (R, -)$. Hence $\Tor_n (R/(aR+D), -)=0$. But
using Lemma \ref{lemma7}$(iii)$, we obtain that $\wgd_R(R/(aR+D))=n$, a
contradiction.

{\bf Case 2.} $R$ satisfies Property (\ref{hypo1}).

Consider the short exact sequence $0\to aR'/a\mathfrak{m'}\to R'/a\mathfrak{m'} \to
R'/aR'\to 0$. From the corresponding long exact sequence of $\Tor$ groups, consider
the following segment
 \[\Tor_n(R' / a\mathfrak{m'} , -) \to \Tor_{n}(R' / aR', -) \to \Tor_{n-1}(aR' /
a\mathfrak{m'},-)\,.\]
  Applying Lemma \ref{lemma8}, we obtain that  $\Tor_n(R' / a\mathfrak{m'} ,-)=0$.
Observing that $aR' / a\mathfrak{m'}\cong R/\mathfrak{m}$ and using Lemma
\ref{lemma1.2} yields
$\Tor_{n-1}(aR' / a\mathfrak{m'}, -)=0$. Hence $\Tor_{n}(R' / aR', -)=0$. But using
Lemma \ref{lemma7} $(iii)$, we obtain that $\wgd_R(R' / aR')=n$, a contradiction.
\end{proof}

\section{Conjecture}
In this section, we restate  \cite[Theorem 6.4]{BG} with an additional hypothesis
and prove the theorem under this additional hypothesis. We also give an example to
show that the proof of Theorem 6.4 as given in \cite{BG} is not complete. We need
the next lemma to give a proof of the modification of \cite[Theorem 6.4]{BG}. We can
use the same idea to give a proof of our Main Theorem.

\begin{lemma}\label{lemma5.1}
Let $R$ be a local Gaussian ring with nilradical $\mathcal{N}$. If $\mathcal{N}\neq
D$, then the maximal ideal of $R_{\mathcal{N}}$ is non-zero.
\end{lemma}
\begin{proof}
Using Theorem \ref{thm2.1}, it follows that the nilradical $\mathcal{N}$ is the
unique minimal prime ideal of $R$. Thus the maximal ideal and the nilradical of
$R_{\mathcal{N}}$ coincide and let us denote it by $\mathcal{N'}$. We want to show
that $\mathcal{N'}\neq 0$. Towards that end, let $x\in \mathcal{N}\setminus D$. We
will show that $0\neq \frac{x}{1}\in R_\mathcal{N}$. Suppose not, then there exists
$y\in R\setminus \mathcal{N}$ such that $xy=0$. Using Theorem \ref{thm2.2}$(iv)$, it
follows that $x^2=0$ or $y^2=0$, a contradiction.
\end{proof}

Noting that the nilpotency degree of an ideal $I$ of $R$ is the smallest $k\in
\mathbb{N}$ such that $I^k=0$, we now restate and prove Theorem 6.4 of \cite{BG}.

\begin{theorem}\label{thm5.1} Let $R$ be a Gaussian ring admitting a maximal ideal
$\mathfrak{m}$ such that the nilradical $\mathcal{N}$ of the localization
$R_{\mathfrak{m}}$ is a non-zero nilpotent ideal. If the nilpotency degree of
$\mathcal{N}\geq 3$, then $\wgd(R)=\infty$.
\end{theorem}
\begin{proof}
Let $\mathfrak{m}$ be a maximal ideal of $R$ such that $R_{\mathfrak{m}}$ has a
non-zero nilpotent nilradical $\mathcal{N}$. Using Theorem \ref{thm2.1}, it follows
that $\mathcal{N}$ is the unique minimal prime ideal. Recall that the
Nilradical($S^{-1}R$)=$S^{-1}$(Nilradical($R$)) for any multiplicative closed set
$S\subset R$. Hence $\mathcal{N}=S^{-1}\mathfrak{n}$, where $\mathfrak{n}$ is the
nilradical of $R$ and $S=R\setminus \mathfrak{m}$ is a multiplicatively closed set
in $R$. Furthermore $\mathfrak{n}$ is a prime ideal of $R$. It is clear that the
maximal ideal of $R_{\mathfrak{n}}$ is nilpotent. It follows from Lemma
\ref{lemma5.1} that the maximal ideal of $R_{\mathfrak{n}}$ is non-zero. The rest of
the proof follows \cite[Theorem 6.4]{BG} mutatis mutandis.
\end{proof}
\begin{Remark}
The hypotheses that the nilpotency degree of $\mathcal{N}\geq 3$ in the above
Theorem ensures that $\mathcal{N}\neq D$.
\end{Remark}

We now give an example to show that the hypothesis on the nilpotency degree in
Theorem \ref{thm5.1} is necessary for the conclusion of Lemma \ref{lemma5.1} to
hold.

\begin{Example}
let $\mathbf{k}$ be a field and $\mathbf{k}[X,Y]$ be a polynomial ring in two
variables. Consider a set $S\subset \mathbf{k}[X,Y]/(XY, Y^2)$ defined by
\[
S=\{a+bY+a_1X+ a_2X^2+\cdots+a_nX^n , \;a,b,a_i\in \mathbf{k}, a\neq 0, n\geq 0\}.
\]
Then $S$ is multiplicative set in $\mathbf{k}[X,Y]/(XY, Y^2)$. Define $R=S^{-1}\big(
\mathbf{k}[X,Y]/(XY, Y^2)\big)$. It is easy to see that the unique maximal ideal of
$R$ is given by $\mathfrak{m}=\{xf(x)+b_1y \mid f(x)\in \mathbf{k}[x]\; and\; b_1\in
\mathbf{k}\}$, where $x, y$ are the images of $X, Y$ in $R$. Any $c,d\in
\mathfrak{m}$ has the form, $c=\lambda_1 y+c_1x+c_2x^2+\cdots c_nx^n$ and
$d=\lambda_2 y+d_1x+d_2x^2+\cdots +d_mx^m$, where $m,n\in \mathbb{N}$ and $c_i,
d_j\in \mathbf{k}$ for $1\leq i\leq n$ and $1\leq j\leq m$. Let $i,j\in \mathbb{N}$
be such that $c_i\ne 0$ and $d_j\neq 0$ and let $i,j$ be the least natural numbers
with this property. Then one can rewrite $c,d$ as $c=\lambda_1
y+x^i(c_i+c_{i+1}'x+\cdots+c_{n}'x^{n-i})$ and $d=\lambda_2
y+x^j(d_j+d_{j+1}'x+\cdots+d_{m}'x^{m-j})$, where $c_{t}'=c_{t+1}c_{i}^{-1}$ and
$d_{s}'=d_{s+1}d_{j}^{-1}$ for all $i\leq t\leq n-i$ and $j\leq s\leq n-j$. Observe
that $c_i+c_{i+1}'x+\cdots+c_{n}'x^{n-i}$ and $d_j+d_{j+1}'x+\cdots+d_{m}'x^{m-j}$
are units in $R$. Furthermore $(c,d)^2=(c^2,cd,d^2)$, $c^2=x^{2i}u_1^2$,
$d^2=x^{2j}u_{2}^2$ and $cd=x^{i+j}u_1u_2$, where $u_1,u_2\notin \mathfrak{m}$. Now
using Theorem \ref{thm2.2}$(iv)$, it can be verified that $R$ is a local Gaussian
ring. Its nilradical $\mathfrak{n}=(y)\subset R$ is not trivial, while the
nilradical of $R_\mathfrak{n}$ is trivial as $\frac{y}1=\frac{xy}x=0$.
\end{Example}

We now prove our Main Theorem.

\begin{theorem}[Main Theorem]\label{main}
Let $R$ be a non-reduced local Gaussian ring with nilradical $\mathcal{N}$. If
$\mathcal{N}\neq D$, then $\wgd (R)=\infty$.
\end{theorem}
\begin{proof} By Theorem \ref{thm2.1}, the nilradical $\mathcal{N}$ is the unique
minimal prime ideal. Thus the maximal ideal and the nilradical of $R_{\mathcal{N}}$
coincide and let us denote it by $\mathcal{N'}$. Since $\wgd(R)\geq
\wgd(R_{\mathcal{N}})$, it suffices to show that $\wgd (R_\mathcal{N})=\infty$.
Using Lemma \ref{lemma5.1}, it follows that $\mathcal{N'}\neq 0$. Hence we have a
local Gaussian ring $(R_\mathcal{N}, \mathcal{N'})$ with $\mathcal{N'}\neq 0$. If
$\mathcal{N}$ is nilpotent, then Theorem \ref{thm5.1} implies that $\wgd
(R_\mathcal{N})=\infty$. If $\mathcal{N}$ is not nilpotent, then Theorem \ref{thm1}
implies that $\wgd (R_\mathcal{N})=\infty$.
\end{proof}

We claim that to prove the Bazzoni-Glaz Conjecture, it remains to consider the case
of a local Gaussian ring with nilradical $\mathfrak{n}=D\neq 0$.
Let $R$ be a Gaussian ring (not necessarily local) and let
$\mathfrak{n}_{\mathfrak{p}}$ denote the nilradical of $R_{\mathfrak{p}}$ for any
$\mathfrak{p}\in \Spec (R)$. We have the
following cases:

\begin{itemize}
  \item[(i)]  $R_{\mathfrak{p}}$ is domain for all $\mathfrak{p}\in \Spec (R)$;
 \item[(ii)]  there exists a $\mathfrak{p}\in \Spec (R)$ such that the
$\mathfrak{n}_{\mathfrak{p}}\neq 0$ and $\mathfrak{n}^2_{\mathfrak{p}}\neq 0$;
   \item[(iii)]  there exists a $\mathfrak{p}\in \Spec (R)$ such that
$\mathfrak{n}_{\mathfrak{p}}\neq 0$ and $\mathfrak{n}^2_{\mathfrak{p}}= 0$.
\end{itemize}

We remind the reader that if $R_{\mathfrak{p}}$ is not a domain, then
$\mathfrak{n}_{\mathfrak{p}}\neq 0$ and hence all possible cases are listed above.
In case (i) $\wgd (R)\leq 1$, while in case (ii) $\wgd (R)= \infty$. Hence to prove
the Bazzoni-Glaz Conjecture it remains to show the following conjecture.

\begin{Conjecture}
Let $R$ be a non-reduced local Gaussian ring with nilradical $\mathcal{N}$. If
$\mathcal{N}^2=0$, then $\wgd (R)=\infty$.
\end{Conjecture}

\end{document}